\documentclass[12pt, a4paper]{article}
 \usepackage{amsmath}
 \usepackage{amsthm}
 \newtheorem{thm}{Theorem}

\begin{document}

\title{Further Results on the Classification of MDS Codes}
\author{%
Janne I. Kokkala\thanks{Supported by Aalto ELEC Doctoral School and Nokia Foundation}\; and Patric R. J. \"Osterg\aa rd\\
\hspace*{5mm}\\
Department of Communications and Networking\\
Aalto University School of Electrical Engineering\\
P.O.\ Box 13000, 00076 Aalto, Finland
}
\date{}

\maketitle

\begin{abstract}
A $q$-ary maximum distance separable (MDS) code $C$ with length $n$, dimension $k$ over an alphabet $\mathcal{A}$ of size $q$ is a set of $q^k$ codewords that are elements of $\mathcal{A}^n$, such that the Hamming distance between two distinct codewords in $C$ is at least $n-k+1$. Sets of mutually orthogonal Latin squares of orders $q\leq 9$, corresponding to two-dimensional \mbox{$q$-}ary MDS codes, and $q$-ary one-error-correcting MDS codes for $q\leq 8$ have been classified in earlier studies. These results are used here to complete the classification of all $7$-ary and $8$-ary MDS codes with $d\geq 3$ using a computer search.
\end{abstract}

\section{Introduction}

A \emph{$q$-ary} \emph{code} $C$ of \emph{length} $n$, and \emph{size} $M$ is a set of $M$ elements, called \emph{codewords}, of $\mathcal{A}^n$, where $\mathcal{A}$ is an alphabet of size $q$. The \emph{minimum distance} $d$ of a code is the smallest Hamming distance between any two distinct codewords. A code with these parameters is called an $(n,M,d)_q$ code. If $\mathcal{A}$ is a finite field and $C$ is a vector subspace, then $C$ is called \emph{linear}. A code that is not linear is called \emph{nonlinear}. Codes that can be either linear or nonlinear are called \emph{unrestricted}.

In the unrestricted case, two codes are called \emph{equivalent} if one can be obtained from the other by a permutation of coordinates followed by permutations of symbols at each coordinate separately. These operations preserve the Hamming distances between codewords. We say that an $(n,M,d)_q$ code is \emph{unique} if all codes with the same parameters are equivalent.

An upper bound for the size of an $(n,M,d)_q$ code is the Singleton bound, 
\[
M \leq q^{n-d+1}.
\]
Codes meeting this bound are called maximum distance separable (MDS), and $k=n-d+1$ is called the \emph{dimension} of the MDS code. A $q$-ary MDS code with length $n$ and dimension $k$ is called an $(n,k)_q$ MDS code.  MDS codes have the property that, given any $k$ coordinates, each $k$-tuple of symbols from $\mathcal{A}$ occurs in the given coordinates in exactly one codeword. It has been conjectured (the MDS conjecture) that an MDS code with parameters $n$, $k$, and prime power $q$ with $1 < k < n-1$ exists if and only if $n \leq q+1$, with the exception that when $q$ is a power of two, MDS codes with $n=q+2$, and $k=3$ or $k=q-1$ exist \cite{S55b}. The conjecture has been proved for linear codes when $q$ is a prime by Ball \cite{B12} and when $q$ is a power of a prime $p$ and $k < 2p-2$ by Ball and De Beule \cite{BD12}.

MDS codes with $d=1$ or $k=1$ are unique and they are called \emph{trivial}; the first contains the whole space $\mathcal{A}^n$ and the latter is a repetition code. The case $k=2$ corresponds to sets of mutually orthogonal Latin squares, which have been classified for $q\leq 9$ \cite{EW14}. MDS codes with $d=2$, dimension $k$ and alphabet size $q$ correspond to $k$-dimensional Latin hypercubes of order $q$. For $q=2,3$, they are trivially unique. Potapov and Krotov \cite{PK11} give a recursive formula for the number of $(k+1,k)_4$ MDS codes but do not classify them up to equivalence. With small $k$ and $q$, $(k+1,k)_q$ MDS codes have been classified by McKay and Wanless \cite{MW08}. 

For $q=2$, nontrivial MDS codes with $d>2$ do not exist. For $q=3$, the only nontrivial MDS code with $d>2$ is the unique $(4,2)_3$ MDS code. Alderson \cite{A06} showed that the $(6,3)_4$ and the $(5,3)_4$ MDS codes are unique. The nonexistence of pairs of mutually orthogonal Latin squares of order $6$ implies the nonexistence of nontrivial $6$-ary MDS codes with $d\geq 3$. For $q=5,7$ all MDS codes with $d\geq 3$, except the $(4,2)_7$ codes, are equivalent to linear codes \cite{KKO14}. For $q=5$, they are unique, which follows from the uniqueness in terms of the notion of equivalence of linear codes \cite{BBFKKW06}. For $q=7,8$, the MDS codes with $d=3$ were classified in \cite{KKO14,KO14}. For $q=7$, some classification results exist for codes with $d>3$ in terms of different notions of equivalence for linear codes; see for example \cite{K06}.

In this work, we classify all $7$-ary $8$-ary MDS codes with $n>3$, $d>3$ by a computer search to finish the classification of all $7$-ary and $8$-ary MDS codes with $d\geq 3$. In Section~\ref{sec:prel}, we discuss basic properties of MDS codes and computational tools used in this work. Section~\ref{sec:gener} explains the computer search used, and the results are given in Section~\ref{sec:res}. Finally, in Section~\ref{sec:d2}, we discuss the case $d=2$ corresponding to Latin hypercubes for small $q$.

\section{Preliminaries} \label{sec:prel}

\subsection{MDS codes}

The operations maintaining equivalence of unrestricted codes form a group of order $(q!)^n n!$, denoted by $G_n$, that acts on $\mathcal{A}^n$. The stabilizer of a code $C$ under this action is called the automorphism group of the code, $\mathrm{Aut}(C)$.

Let $C$ be an $(n,k)_q$ MDS code. For a symbol $i \in \mathcal{A}$, we denote by $C_i$ the $(n-1,k-1)_q$ MDS code obtained by removing the last coordinate and retaining the codewords that have symbol $i$ in that coordinate,
\[
C_i = \{ (c_1,c_2,\dots,c_{n-1}) \,:\, c \in C, c_n = i\}.
\]

Removing a symbol at a given position of each codeword of $C$ yields an $(n-1,k)_q$ MDS code $C'$. This operation is called \emph{puncturing}. We say that $C$ is an \emph{extension} of $C'$. An $(n,k)_q$ MDS code $C$ is \emph{extendable} if there exists an extension of $C$ that is an $(n+1,k)_q$ MDS code. There is a one-to-one correspondence between extensions of $C$ (up to the position of the new coordinate) and labeled partitions of $C$ into $(n,k-1)_q$ MDS codes: $\bigcup_{i \in \mathcal{A}} C^i i$, where $C^ii$ denotes the code obtained by adding the symbol $i$ at the end of each codeword in $C^i$, is an extension of $C$ if and only if the sets $C^i$ form a partition of $C$ into $(n,k-1)_q$ MDS codes.

It is known that if a linear code is extendable, it has an extension that is linear \cite{AG09}. For fixed $q$ and $d$ and sufficiently large $n$, an extension of a linear MDS code is necessarily equivalent to a linear MDS code \cite{ABS07}. However, very little is known about extending nonlinear codes.

\subsection{Tools}

We reduce the problem of detecting code equivalence to the graph isomorphism problem and use the software \emph{nauty} \cite{MP14} for solving the instances. For an $(n,M,d)_q$ code $C$, we define a colored graph as follows: The graph contains $n$ copies of $K_q$, the complete graph of order $q$, named $\Gamma_1,\Gamma_2,\dots,\Gamma_n$, colored with color 1. For each $i$, the vertices in $\Gamma_i$ correspond to the elements in $\mathcal{A}$. For each codeword $c \in C$, there is an additional vertex, colored with color 2, which is adjacent to the vertex corresponding to the symbol $c_i$ in $\Gamma_i$ for each $i\in \{1,2,\dots,n\}$. A graph isomorphism preserving the coloring permutes the copies of $K_q$ and the vertices in each of them separately, corresponding to a permutation of coordinates and permutations of symbols in each coordinate separately in the code, respectively. Two codes are equivalent if and only if their corresponding graphs are isomorphic. Further, the automorphism group of a code corresponds to the automorphism group of the graph. The software \emph{nauty} can be used to find the automorphism group of a graph and all isomorphisms between two graphs. For large graphs, we use \emph{nauty} in the sparse mode with the random Schreier algorithm enabled.

For a finite set $X$ and a family $\mathcal{S}$ of subsets of $X$, an \emph{exact cover} is a subset $\mathcal{S}'$ of $\mathcal{S}$ that partitions $X$. We use the library \emph{libexact} \cite{KP08} for finding all exact covers, given $X$ and $\mathcal{S}$.

Finally, we use \emph{cliquer} \cite{NO03} to find all cliques of given size in a graph.

\section{Generation} \label{sec:gener}

Given a set $\mathcal{C}$ of equivalence class representatives of $(n,k)_q$ MDS codes, every $(n+1,k)_q$ MDS code is equivalent to a code that is an extension of a code in $\mathcal{C}$. This reduces the problem of generating equivalence class representatives of $(n+1,k)_q$ MDS codes to the problem of finding all partitions of all codes in $\mathcal{C}$ into $(n,k-1)_q$ MDS codes. For $k=2$, this is precisely the method of extending a set of MOLS with a new Latin square by partitioning the Latin squares in the set into common transversals, used for example in \cite{EW14, MW04, MMM07, P63}; for discussion of this and similar problems, see for example \cite{O05}. 

Consider a partition of an $(n,k)_q$ MDS code $C$ into $(n,k-1)_q$ MDS codes $C^i$ for $i \in \mathcal{A}$,
\[
C = \bigcup_{i\in \mathcal{A}} C^i.
\]
We have
\[
C_j = \bigcup_{i \in \mathcal{A}} C^i_j
\]
and
\[
C^i = \bigcup_{j\in\mathcal{A}} C^i_jj.
\]

Finding all possible partitions of $C$ into $(n,k-1)_q$ MDS codes can now be reduced to finding all possible partitions $\{C^i_j\}_{i \in \mathcal{A}}$ of $C_j$ for each $j$ and combining those partitions in all possible ways. Instead of solving the reduced problem using recursion, our method for finding the partitions of $(n,k)_q$ MDS codes uses the results from finding the partitions of $(n-1,k-1)_q$ MDS codes. In this work, we generate all $q$-ary MDS codes with $k\geq 3$, $d\geq 4$ starting from the MDS codes with $k=2$, their partitions, and the codes with $d=3$.

\subsection{Algorithm}

As an initial step, we find all partitions of the equivalence class representatives of $(n,2)_q$ MDS codes into $(n,1)_q$ MDS codes in the following way. Let $\mathcal{C}_n$ be the set of representatives of $(n,2)_q$ MDS codes. For a given $n$, we loop over all $\hat{C} \in \mathcal{C}_{n+1}$ and over all $n+1$ punctured codes $C'$ of $\hat{C}$. Now, $\hat{C}$ is an extension of $C'$ so it induces a partition of $C'$ into $(n,1)_q$ MDS codes. We find the equivalence class representative $\hat{C}' \in \mathcal{C}_n$ for which $\hat{C}' \cong C'$ and find all $g \in G_n$ for which $gC' = \hat{C}'$. Considering how the codewords are mapped by $g$, we get partitions of $\hat{C}'$ into $(n,1)_q$ MDS codes. All partitions of each element in $\mathcal{C}_n$ are obtained in this way, as every such partition corresponds to an $(n+1,2)_q$ MDS code $C$ that is equivalent to some $\hat{C} \in \mathcal{C}_{n+1}$.

The algorithm for finding all partitions of an $(n,k)_q$ MDS code $C$ into $(n,k-1)_q$ MDS codes consists of three parts.
\begin{enumerate}
\item For each $j \in \mathcal{A}$, find all $(n-1,k-2)_q$ MDS codes that occur as a part in a partition of $C_j$ into $(n-1,k-2)_q$ MDS codes. Because $C_j \cong \hat{C}$, where $\hat{C}$ is a representative whose partitions into $(n-1,k-2)_q$ MDS codes are known, this can be solved by finding a $g \in G_n$ for which $C_j=g\hat{C}$ and applying $g$ directly to each partition of $\hat{C}$. Denote the set of $(n-1,k-2)_q$ MDS codes found in this step by $\mathcal{S}_j$.
\item Find all subsets $D$ of $C$ that are $(n,k-1)_q$ MDS codes for which $D_j \in \mathcal{S}_j$ for each $j$. This is done by finding all $q$-cliques in the $q$-partite graph where, for each $j$, each element of $\mathcal{S}_j$ corresponds to a vertex in the $j$th part, and between two vertices in different parts there is an edge if the minimum distance between them is at least $n-k+1$.
\item Find all partitions of $C$ into $(n,k-1)_q$ MDS codes. This is done by solving an exact cover problem, where $q$ sets are selected from the family created in the previous phase such that each codeword in $C$ is covered exactly once.
\end{enumerate}

We run the algorithm for one representative $C$ from each equivalence class of $(n,k)_q$ MDS codes, and finally perform isomorph rejection for all $(n+1,k)_q$ MDS codes corresponding to the obtained partitions.

\subsection{Consistency check}

We check the consistency of each phase of the generation by double counting. The total number of $(n+1,k)_q$ MDS codes is the sum of the sizes of the equivalence classes,
\[
\sum_{C \in \mathcal{C}_{n+1}} \frac{|G_{n+1}|}{|\mathrm{Aut}(C)|},
\]
where $\mathcal{C}_{n+1}$ is the obtained set of equivalence class representatives of $(n+1,k)_q$ MDS codes. On the other hand, the same count can be obtained as follows. Let $N(C)$ be the number of partitions of an $(n,k)_q$ MDS code $C$ into $(n,k-1)_q$ MDS codes. The number of $(n+1,k)_q$ MDS codes is
\[
\sum_C N(C) q! = \sum_{\hat{C} \in \mathcal{C}_{n}} \frac{|G_{n}| N(\hat{C})}{|\mathrm{Aut}(\hat{C}) |} q!,
\]
where the first sum is taken over all $(n,k)_q$ MDS codes, $\mathcal{C}_{n}$ is a set containing exactly one representative from each equivalence class of $(n,k)_q$ MDS codes, and $N(\hat{C})$ is the number of partitions of $\hat{C}$ obtained by the algorithm.

\section{Results} \label{sec:res}

The algorithm was run for $q=7$ and $q=8$, for each $k= 3,4,\dots,q-1$, and for each $n$ starting from $n= k+2$ and increasing $n$ until no more codes were found. The numbers of inequivalent extendable codes are shown in Tables~\ref{tab:n_extendable_7} and \ref{tab:n_extendable_8}, for $q=7,8$, respectively, and the numbers of equivalence classes are shown in Tables~\ref{tab:n_results_7} and \ref{tab:n_results_8}. In the tables, the numbers marked with * are obtained in this work using the method described above. The previously known numbers are obtained from the classification of one-error-correcting MDS codes and sets of MOLS, and in addition for $q=7$ the MDS conjecture gives an upper bound for $n$, and the nonexistence of $(10,3)_8$ MDS codes follows form the nonexistence of $(9,2)_8$ MDS codes. The computations lasted a few minutes for $q=7$ and 21 hours for $q=8$; 18 hours were needed for extending the $(9,7)_8$ MDS codes. 

By comparing the equivalence class representatives with known linear codes, we see that all $8$-ary MDS codes with $d\geq 5$ or $k\geq 4,d=4$ are equivalent to a linear code.

\begin{table}
\begin{center}
\begin{tabular}{r|rrrrrr}
$n \,\backslash \, k$ & $2$ &  $3$ & $4$ & $5$ & $6$ & $7$  \\
\hline
$3$ & $6$ \\
$4$ & $2$ & $1$ \\
$5$ & $1$ & *$1$ & $1$ \\ 
$6$ & $1$ & *$1$ & *$1$ & $1$ \\
$7$ & $1$ & *$1$ & *$1$ & *$1$ & $1$ \\
$8$ & $0$ & $0$ & $0$ & $0$ & $0$ & $0$ \\
\hline
\end{tabular}
\vspace{0.3cm} \\
\end{center}
\caption{Number of inequivalent extendable nontrivial $(n,k)_7$ MDS codes}
\label{tab:n_extendable_7}
\end{table}

\begin{table}
\begin{center}
\begin{tabular}{r|rrrrrr}
$n \,\backslash \, k$ & $2$ & $3$ & $4$ & $5$ & $6$ & $7$\\
\hline
$3$ & $147$ &\\
$4$ & $7$ & &\\
$5$ & $1$ & $1$ & \\ 
$6$ & $1$ & *$3$ & $1$ &  \\
$7$ & $1$ & *$1$ & *$1$  & $1$ &  \\
$8$ & $1$ & *$1$ & *$1$ & *$1$ & $1$   \\
$9$ & $0$ & $0$ & $0$ & $0$ & $0$ & $0$     \\
\hline
\end{tabular}
\vspace{0.3cm} \\
\end{center}
\caption{Number of equivalence classes of nontrivial $(n,k)_7$ MDS codes}
\label{tab:n_results_7}
\end{table}

\begin{table}
\begin{center}
\begin{tabular}{r|rrrrrr}
$n \,\backslash \, k$ & $2$ &  $3$ & $4$ & $5$ & $6$ & $7$  \\
\hline
$3$ & $2\,024$ \\
$4$ & $38$ & $4\,470$ \\
$5$ & $1$ & *$44$ & $36$ \\ 
$6$ & $1$ & *$2$ & *$1$ & $12$ \\
$7$ & $1$ & *$2$ & *$1$ & *$1$ & $7$ \\
$8$ & $1$ & *$2$ & *$1$ & *$1$ & *$1$ & 4 \\
$9$ & $0$ & *$2$ & *$0$ & *$0$ & *$0$ & *$1$ \\
$10$ & & $0$ & & & & *$0$ \\
\hline
\end{tabular}
\vspace{0.3cm} \\
\end{center}
\caption{Number of inequivalent extendable nontrivial $(n,k)_8$ MDS codes}
\label{tab:n_extendable_8}
\end{table}

\begin{table}
\begin{center}
\begin{tabular}{r|rrrrrrr}
$n \,\backslash \, k$ & $2$ & $3$ & $4$ & $5$ & $6$ & $7$ & $8$ \\
\hline
$3$ & $283\,657$ &\\
$4$ & $2\,165$ & &\\
$5$ & $39$ & $12\,484$ & &\\ 
$6$ & $1$ & *$39$ & $14$ & & \\
$7$ & $1$ & *$2$ & *$2$  & $8$ & &  \\
$8$ & $1$ & *$2$ & *$1$ & *$2$ & $4$ & &  \\
$9$ & $1$ & *$2$ & *$1$ & *$1$ & *$2$ & $4$ &    \\
$10$ & $0$ & *$1$ & *$0$ & *$0$ & *$0$ & *$1$ & $0$   \\
$11$ & & $0$& & & & *$0$ \\
\hline
\end{tabular}
\vspace{0.3cm} \\
\end{center}
\caption{Number of equivalence classes of nontrivial $(n,k)_8$ MDS codes}
\label{tab:n_results_8}
\end{table}

\section{MDS codes with $d=2$} \label{sec:d2}

The preceding part of this work completes the classification of all MDS codes with minimum distance at least $3$ for alphabet sizes at most $8$. In this section, we discuss the remaining case of minimum distance $2$. We use a simple construction to give a lower bound for the number of equivalence classes of $(n,n-1)_q$ MDS codes, or equivalently the paratopy classes of $(n-1)$-dimensional Latin hypercubes of order $q$, for small $q$.

We denote by $N_n$ the number of $(n,n-1)_q$ MDS codes and by $M_n$ the number of equivalence classes of $(n,n-1)_q$ MDS codes. Because each class contains at most $|G_n| = n! (q!)^{n}$ elements, we get
\begin{equation}
M_n \geq \frac {N_n} {n! (q!)^{n}}. \label{eq:lowerboundm}
\end{equation}

For small values of $n$ and $q$, we find lower bounds for $N_n$ and $M_n$ by the following construction. This construction resembles Construction X4 given in \cite[Chapter 18.7]{MS77}. We denote by $CC'$ the direct sum of codes $C$ and $C'$ of lengths $n$ and $n'$, respectively,
\[
CD = \{(c_1,c_2,\dots,c_{n},c'_1,c'_2,\dots,c'_{n'}) \,:\, c \in C, c' \in C' \}.
\]

\begin{thm} \label{thm:constr}
Let $C$ and $C'$ be $(n_1,n_1-1)_q$ and $(n_2,n_2-1)_q$ MDS codes, respectively. Then
\[
C'' = \bigcup_{i\in\mathcal{A}} C_iC'_i
\]
 is an $(n_1+n_2-2,n_1+n_2-3)_q$ MDS code.
\end{thm}
\begin{proof}
The code $C''$ contains $qq^{n_1-2}q^{n_2-2}=q^{n_1+n_2-3}$ codewords of length $n_1+n_2-2$. Because the minimum distance of each $C_i$ and $C'_i$ is $2$, the minimum distance of each $C_iC'_i$ is at least $2$. For distinct $i,j$, any codeword in $C_i C'_i$ differs from any codeword in $C_jC'_j$ in at least one position in the first $n_1-1$ coordinates and in at least one position in the last $n_2-1$ coordinates. Therefore, $C''$ has minimum distance $2$.
\end{proof}

Because the construction in Theorem~\ref{thm:constr} maps exactly $q!$ pairs of $(C,C')$ to the same code, this yields a lower bound for the number of $(n,n-1)_q$ MDS codes,
\begin{equation}
N_{n} \geq \max_{n'} \frac{N_{n'} N_{n-n'+2}}{q!}. \label{eq:lowerboundn}
\end{equation}
This lower bound along with the known numbers of $(n,n-1)_q$ MDS codes and \eqref{eq:lowerboundm} allows us to find lower bounds for the number of equivalence classes of $(n,n-1)_q$ MDS codes for small $n$.

Table~\ref{tab:hypercubesm} lists the known numbers of equivalence classes of $(n,n-1)_q$ MDS codes for $4\leq q\leq 8$ given in \cite{MW08} and lower bounds obtained by \eqref{eq:lowerboundm} and \eqref{eq:lowerboundn} for small values of $n$. Since the lower bounds for the unknown values are rather large, except for a few entries, obtaining further classification results by explicitly constructing representatives from each equivalence class is not feasible in general.

Finally, we note that \cite{PK11} gives a double exponential lower bound for $N_n$ when $q\geq 4$:
\[
\log_2 N_n \geq 
\begin{cases}
\left(\frac{q}{2}\right)^{n-1}, &  \text{for even } q, \\
\left(\frac{(q-3)(q-1)}{4}\right)^{(n-1)/2}, & \text{for odd } q.
\end{cases}
\]
From \eqref{eq:lowerboundm}, it follows that that the number of equivalence classes $M_n$ for $q\geq 4$ also has a lower bound that is double exponential in $n$.

\begin{table}
\begin{center}
\begin{tabular}{r|rrrrr}
$n \,\backslash\, q$ & $4$ & $5$ & $6$ & $7$ & $8$ \\
\hline
$3$ & $2$ & $2$ & $12$ & $147$ & $283\,657$ \\ 
$4$ & $5$ & $15$ & $264\,248$ & $\geq 4.8 \times 10^{7}$ & $\geq 4.6 \times 10^{15}$ \\ 
$5$ & $26$ & $86$ & $\geq 4.8 \times 10^{7}$ & $\geq 2.3 \times 10^{13}$ & $\geq 6.2 \times 10^{25}$ \\ 
$6$ & $4\,785$ & $3\,102$ & $\geq 1.4 \times 10^{13}$ & $\geq 9.5 \times 10^{18}$ & $\geq 6.9 \times 10^{35}$ \\ 
$7$ & $\geq 3.5 \times 10^{9}$ & $\geq 1.5 \times 10^{3}$ & $\geq 3.1 \times 10^{15}$ & $\geq 3.3 \times 10^{24}$ & $\geq 6.6 \times 10^{45}$ \\ 
\end{tabular}
\vspace{0.3cm} \\
\end{center}
\caption{Number of equivalence classes of $(n,n-1)_q$ MDS codes}
\label{tab:hypercubesm}
\end{table}

\end{document}